\documentclass[letterpaper,11pt]{amsart}
\usepackage[inner=2.3cm,outer=2.3cm, bottom=3.3cm]{geometry}
\usepackage{amsmath,amsgen,amstext,amsbsy,amsopn,amsfonts, enumitem, amssymb, amsthm, comment}  
\usepackage{url,graphicx,tabularx,array, tikz-cd, setspace}
\definecolor{leaf}{rgb}{0,.35,0}
\definecolor{chianti}{rgb}{0.6,0,0}
\definecolor{meretale}{rgb}{0,0,.6}
\usepackage[colorlinks=true,linkcolor=meretale,
urlcolor=chianti, citecolor=leaf,pagebackref]{hyperref}

\numberwithin{equation}{subsection}

\newcommand{\Hom}{\operatorname{Hom}}

\newcommand{\depth}{\operatorname{depth}}
\newcommand{\sD}{\operatorname{\mathsf{D}}}

\newcommand{\up}[1]{\textup{#1}}
\newcommand{\rb}{\mathsf{b}}
\newcommand{\rf}{\mathsf{f}}
\newcommand{\Dbf}{\sD_\rf^\rb}
\newcommand{\Db}{\sD_\rb}
\newcommand{\mx}{\mathfrak{m}}

\newcommand{\HH}{\up{H}}
\newcommand{\Z}{\up{Z}}
\newcommand{\B}{\up{B}}
\newcommand{\C}{\up{C}}
\newcommand{\Proj}{\mathsf{P}}
\newcommand{\Inj}{\mathsf{I}}
\newcommand{\Flat}{\mathsf{F}}
\newcommand{\GP}{\mathsf{GP}}
\newcommand{\GI}{\mathsf{GI}}
\newcommand{\GF}{\mathsf{GF}}
\newcommand{\Cl}{\mathsf{C}}

\newcommand{\id}{\operatorname{id}}
\newcommand{\pd}{\operatorname{pd}}
\newcommand{\fd}{\operatorname{fd}}
\newcommand{\Gpd}{\operatorname{Gpd}}
\newcommand{\Gid}{\operatorname{Gid}}
\newcommand{\Gfd}{\operatorname{Gfd}}
\newcommand{\Cld}{\operatorname{\Cl-dim}}
\newcommand{\Ext}{\operatorname{Ext}}

\newcommand{\level}[1]{\operatorname{level}_{R}^{#1}}
\newcommand{\thick}{\operatorname{thick}_R}

\newtheorem{theorem}{Theorem}[section]
\newtheorem{corollary}{Corollary}[theorem]
\newtheorem{lemma}{Lemma}[section]
\newtheorem*{varthmA}{Theorem A}
\newtheorem*{varthmB}{Theorem B}
\newtheorem*{varthmC}{Theorem C}
\newtheorem*{varthmD}{Theorem D}

\theoremstyle{definition}
\newtheorem{note}{}[section]
\newtheorem{defn}{Definition}[section]
\newtheorem{example}{Example}[section]

\title{Level Inequalities for Complexes}
\author{Zachary Nason}
\address{Department of Mathematics, University of Nebraska, Lincoln, NE 68588-0130, USA}
\email{znason2@huskers.unl.edu}

\begin{document}
\begin{abstract}
We prove that for all noetherian rings, the level of any complex $M \in \Db(R)$ with respect to the collection of projective or injective modules is bounded above by $\pd(\HH(M)^\oplus) + 1$ or $\id(\HH(M)^\oplus) + 1$, respectively. In addition, we also prove that if $\Cl$ is the collection of flat, Gorenstein projective, Gorenstein injective, or Gorenstein flat modules, then the level of any complex $M \in \Db(R)$ with respect to $\Cl$ is bounded above by $\max\{2, \Cld(\HH(M)^\oplus) + 1\}$. These results give universal bounds for the projective, injective, and flat levels over regular local rings, and give universal bounds for the Gorenstein projective, Gorenstein injective, and Gorenstein flat levels over Gorenstein local rings. As an application of the above results, we prove a version of the Bass Formula for complexes with respect to injective level and Gorenstein injective level. We also show that the bounds achieved for each homological and Gorenstein homological level considered is optimal.
\end{abstract}
\maketitle
\section{Introduction}
We assume that $R$ is a commutative noetherian ring. In the derived category of $R$, the level of $M \in \Db(R)$ with respect to a collection of objects $\Cl$ (referred to as the $\Cl$-level) is the number of mapping cones involving objects in $\Cl$ necessary to obtain $M$. The $\Cl$-level of an object $M$ with respect to a collection of objects $\Cl$ can give a wealth of information concerning both $M$ and the overarching ring $R$. As an example, when considering the collection of projective modules $\Proj$, the $\Proj$-level of $M$ gives similar information as the classical projective dimension of $M$. In particular, if $R$ is local, then $R$ is regular if and only if the $\Proj$-level of all homologically bounded complexes is finite (see for instance \cite[Lemma 1.2, Corollary 2.2]{LowerBound}). Obtaining universal upper and lower bounds for the $\Cl$-levels of complexes in $\Db(R)$ is an area of continuing research interest. In \cite[Theorem 5.5]{HomPerf}, Avramov, Buchweitz, Iyengar, and Miller establish that for an complex $M \in \Db(R)$ with finitely generated homology over a Noetherian ring, the $\Proj^\mathsf{f}$-level of $M$ (where $\Proj^\mathsf{f}$ is the collection of all finitely generated projective modules) is bounded above by the projective dimension of $\HH(M)^\oplus$ plus one. Altmann, Grifo, Monta\~{n}o, Sanders, and Vu establish in \cite{LowerBound} lower bounds for the $\Proj$-level using the Ghost Lemma. Expanding beyond the collection of projective modules, Awadalla and Marley establish in \cite{LevelGor} both lower and upper bounds for the level with respect to Gorenstein projective modules. Recently, Christensen, Kekkou, Lyle, and Soto Levins \cite{GPerf} have optimized the upper bound for the Gorenstein projective level obtained in \cite{LevelGor} for complexes with finitely generated homology. In particular, they have shown that for all $M \in \Dbf(R)$, the Gorenstein projective level of $M$ is bounded above by $\max\{2, \Gpd_R(\HH(M)^\oplus) + 1\}$. In order to obtain this bound, the authors of \cite{GPerf} used the Auslander-Bridger Formula, which allowed them to precisely control the Gorenstein projective dimension of a complex with the depth of the complex. In this paper, we extend the arguments made in \cite{GPerf} in order to obtain bounds of the Gorenstein projective level of all bounded complexes in $\sD(R)$, including complexes with non-finitely generated homology. In addition, the arguments we use can be applied more generally to obtain bounds for the flat level, Gorenstein injective level, and Gorenstein flat level of a bounded complex in $\sD(R)$. We obtain the following result.

\begin{varthmA} 
Let $M \in \Db(R)$, and let $\Cl$ be the class of flat, Gorenstein projective, Gorenstein injective, or Gorenstein flat modules. Then the following inequality holds:
\begin{equation*}
	\level{\Cl}(M) \leq \max\{2, \Cld(\HH(M)^\oplus) + 1\}
\end{equation*}
where $\Cld$ represents flat dimension, Gorenstein projective dimension, or Gorenstein flat dimension, respectively.
\end{varthmA}

We also find stronger bounds for the projective and injective levels of any bounded complex.

\begin{varthmB}
Let $M \in \Db(R)$, and let $\Proj$ be the collection of all projective $R$-modules and $\Inj$ be the collection of all injective modules. The following inequalities holds:
\begin{align*}
	\level{\Proj} M &\leq \pd_R\HH(M)^\oplus + 1 \\
	\level{\Inj} M &\leq \id_R\HH(M)^\oplus + 1
\end{align*}
\end{varthmB}

As an immediate consequence of these theorems, we obtain upper bounds for the projective, injective, and flat levels over regular local rings, and upper bounds for the Gorenstein projective, Gorenstein injective, and Gorenstein flat levels over Gorenstein local rings.

As an additional application of the above results, we are able to obtain a version of the Bass Formula for injective levels.
\begin{varthmC}
Let $R$ be a noetherian local ring, and let $M \in \Dbf(R)$ such that $\id(\HH(M)^\oplus)) < \infty$. We have that
\begin{equation*}
\level{\Inj}(M) = \depth(R)+1
\end{equation*}
\end{varthmC}

In addition, we obtain a weaker version of the Bass Formula for Gorenstein injective levels.
\begin{varthmD}
Let $R$ be a noetherian local ring of positive depth, and let $M \in \Dbf(R)$ such that $\Gid(\HH(M)^\oplus) < \infty$. We have that
\begin{equation*}
\level{\GI}(M) = \depth(R) + 1
\end{equation*}
\end{varthmD}

At the end of the paper, we demonstrate that the bounds obtained for each of the homological and Gorenstein homological levels considered is optimal.

\section{Background}
\subsection{The Derived Category and Levels}
All $R$-complexes $M$ (that is, chain complexes of $R$-modules) will be graded homologically.
\begin{equation*}
M = \cdots \longrightarrow M_{n+1} \xrightarrow{\partial_{n+1}} M_n \xrightarrow{\partial_n} M_{n-1} \longrightarrow \cdots
\end{equation*}
For every $n \in \mathbb{Z}$, we let $\Z_n(M) = \ker(\partial_n)$, $\B_n(M) = \operatorname{im}(\partial_{n+1})$, $\HH_n(M) = \Z_n(M)/\B_n(M)$, and $\C_n(M) = M_n/\B_n(M)$. We denote the homological supremum of $M$ as $\sup(M)$, and the homological infimum as $\inf(M)$. It is important to note that these are not the graded supremum or infimum of $M$.

We work in the derived category of a ring, denoted $\sD(R)$. This category is obtained from the category of chain complexes by formally inverting all quasi-isomorphisms. Unlike the category of chain complexes, the derived category is rarely abelian. Instead, the derived category is a triangulated category, where distinguished triangles take the place of short exact sequences. As with any triangulated category, it is possible to define the level of a complex with respect to some collection of objects in $\sD(R)$. Intuitively, the level of a complex is the smallest number of mapping cones needed to build the complex with respect to the given collection of objects in $\sD(R)$. A formal definition is given below.

\begin{defn}
Let $\mathsf{C}$ be a non-empty collection of objects in $\sD(R)$. We recursively define the $n$th thickening of $\mathsf{C}$ as follows:
\begin{enumerate}
	\item $\thick^0(\mathsf{C})$ is the zero object
	\item $\thick^1(\mathsf{C})$ is all objects in $\sD(R)$ that are isomorphic to shifts, direct summands, or finite direct sums of objects in $\Cl$.
	\item For $n \geq 1$, $\thick^n(\mathsf{C})$ is all objects $M \in \sD(R)$ such that there exists a distinguished triangle
	\begin{equation*}
		K \longrightarrow L \oplus M \longrightarrow N \longrightarrow
	\end{equation*}
	with $K \in \thick^1(\mathsf{C})$ and $N \in \thick^{n-1}(\mathsf{C})$. 
\end{enumerate}

Let $M$ be an object in $\sD(R)$. We now define the level of $M$ with respect to $\mathsf{C}$ as follows:
\begin{equation*}
	\level{\mathsf{C}}(M) = \inf\{n \, | \, M \in \thick^{n}(\mathsf{C})\}
\end{equation*}
\end{defn}

There are a few basic lemmas that are commonly used to work with levels, and for completeness they are listed below.

\begin{lemma}\label{BasicLemma}
Let $\Cl$ be a non-empty collection of objects in $\sD(R)$. Let $L$, $M$, and $N$ be objects in $\sD(R)$.
\begin{enumerate}
	\item $\level{\mathsf{C}}(M) = \level{\mathsf{C}}(N)$ if $M \simeq N$.
	\item $\level{\mathsf{C}}(M) = \level{\mathsf{C}}(\Sigma^n M)$ for all $n \in \mathbb{Z}$
	\item If $L \longrightarrow M \longrightarrow N \longrightarrow$ is a distinguished triangle, then
	\begin{equation*}
		\level{\mathsf{C}}(M) \leq \level{\mathsf{C}}(L) + \level{\mathsf{C}}(N)
	\end{equation*}
	\item $\level{\mathsf{C}}(M \oplus N) = \max\{\level{\mathsf{C}}(M), \level{\mathsf{C}}(N)\}$
\end{enumerate} 
\end{lemma}
\begin{proof}
See \cite[Lemma 2.4]{HomPerf}
\end{proof}

\subsection{Adams Resolutions}
The concept of an Adams resolution of an object in a triangulated category (containing enough projectives/injectives) was first developed by Christensen in \cite{IdealTriang}. We will use Adams resolutions of objects in $\Db(R)$ as described in \cite[Construction 3.1]{GPerf}. We first construct an projective Adams resolution  of $M \in \Db(R)$ in the following way. In $M$, choose a (possibly infinite) set of cycles $\{z_i\}_{i \in \Lambda}$ such that their homology classes generate $\HH(M)$. For each $z_i$ in this set, consider the complex $F_i = \Sigma^{|z_i|} R$ and the associated morphism $\phi_i: F_i \to M$ that sends $1$ to $z_i$. Let $F^0 = \bigoplus\limits_{i \in \Lambda} F_i$ and $\phi^0: F^0 \to M$ the morphism inherited from $\{\phi_i\}_{i \in \Lambda}$. We set $\Omega^1(M) = \Sigma^{-1}\text{Cone}(\phi^0)$, which yields the distinguished triangle
\begin{equation*}
	\Omega^1(M) \longrightarrow F^0 \overset{\phi^0}{\longrightarrow} M \longrightarrow
\end{equation*}
We recursively define $F^n$ to be a graded-free module that maps onto the generating cycles of $\Omega^n(M)$ using the same process as in the previous paragraph, and set $\Omega^{n+1}(M) = \Omega^1(\Omega^n(M))$. For all $n \in \mathbb{N}$, we have the distinguished triangle
\begin{equation*}
	\Omega^{n+1}(M) \longrightarrow F^n \overset{\phi^n}{\longrightarrow} \Omega^{n}(M) \longrightarrow
\end{equation*}

Since $F^n$ surjects onto $\HH(\Omega^n(M))$ by construction, we also have the short exact sequences
\begin{equation*}
	0 \longrightarrow \HH(\Omega^{n+1}(M)) \longrightarrow F^n \xrightarrow{\HH(\phi^n)} \HH(\Omega^n(M)) \longrightarrow 0
\end{equation*}
for all $n \in \mathbb{N}$.

In addition, we will construct an injective Adams resolution of $M \in \Db(R)$ in a dual manner to the previous construction. Consider the character complex $\Hom_R(M, \mathbb{E})$ of $M$ (where $\mathbb{E}$ is a faithfully injective $R$-module). Now construct the complex $F^0$ for the complex $\Hom_R(M, \mathbb{E})$ in the same process as above, and let $I^0 = \Hom_R(F^0, \mathbb{E})$. $I^0$ is a bounded complex of injective modules with zero differential, and there exists a map $\iota^0: M \to I^0$ which is composition of the biduality map $M \to \Hom_R(\Hom_R(M, \mathbb{E}), \mathbb{E})$ with the map $\Hom_R(\Hom_R(M, \mathbb{E}), \mathbb{E}) \to \Hom_R(F^0, \mathbb{E})$.  By construction, the map on homology $F^0 \to \HH(\Hom_R(M, \mathbb{E}))$ is surjective, and so the map $\HH(\Hom_R(\Hom_R(M, \mathbb{E}), \mathbb{E}) \to I^0$ is injective. The biduality morphism on homology $\HH(M) \to \Hom_R(\Hom_R(\HH(M), \mathbb{E}), \mathbb{E})$ is injective \cite[Proposition 4.5.3]{DCMCA}, and we have $\Hom_R(\Hom_R(\HH(M), \mathbb{E}), \mathbb{E}) \cong \HH(\Hom_R(\Hom_R(M, \mathbb{E}), \mathbb{E}))$ since $\mathbb{E}$ is injective. This forces $\HH(\iota^0): \HH(M) \to I^0$ to be injective. This yields the distinguished triangle
\begin{equation*}
	M \overset{\iota^0}{\longrightarrow} I^0 \longrightarrow \Theta^1(M) \longrightarrow
\end{equation*}
where $\Theta^1(M)$ is defined as the mapping cone of $\iota^0$. We also have the short exact sequence
\begin{equation*}
	0 \longrightarrow \HH(M) \xrightarrow{\HH(\iota^0)} I^0 \longrightarrow \HH(\Theta^1(M)) \longrightarrow 0
\end{equation*}

Through iteration in the same manner as the projective Adams resolution, we have distinguished triangles
\begin{equation*}
	\Theta^n(M) \overset{\iota^n}{\longrightarrow} I^n \longrightarrow \Theta^{n+1}(M) \longrightarrow
\end{equation*}
for all $n \in \mathbb{N}$ where $I^n$ is a bounded complex of injective modules with zero differential, and $\Theta^{n+1}(M) = \Theta^1(\Theta^n(M))$. These distinguished triangles also induce short exact sequences
\begin{equation*}
	0 \longrightarrow \HH(\Theta^n(M)) \xrightarrow{\HH(\iota^n)} I^n \longrightarrow \HH(\Theta^{n+1}(M)) \longrightarrow 0
\end{equation*}

We will need two lemmas first established in \cite{GPerf} that allow us to effectively bound levels of complexes and homological dimensions using Adams resolutions.

\begin{lemma}\label{SyzygyInequality}
Let $M \in \Db(R)$. Suppose that $\mathsf{F}$ be a collection of objects in $\sD(R)$ that contains all free modules, and suppose that $\mathsf{I}$ is a collection of objects in $\sD(R)$ that contains all injective modules. Then we have the inequalities
\begin{align*}
	\level{\mathsf{F}}(M) &\leq \level{\mathsf{F}}\Omega^n(M) + n \\
	\level{\mathsf{I}}(M) &\leq \level{\mathsf{I}}\Theta^n(M) + n 
\end{align*}
\end{lemma}
\begin{proof}
This proof is essentially \cite[Lemma 3.3]{GPerf} but we prove it in full generality. For the first inequality, we only need to show that
\begin{equation*}
	\level{\mathsf{F}}(M) \leq \level{\mathsf{F}}\Omega^1(M) + 1
\end{equation*}
and to do this, we need only look at the distinguished triangle
\begin{equation*}
	\Omega^1(M) \longrightarrow F^0 \overset{\phi^0}{\longrightarrow} M \longrightarrow
\end{equation*}
and recognize that $F^0$ is a bounded complex of free modules with zero differential. Lemma~\ref{BasicLemma} then gives the necessary inequality.
To prove the second inequality, we only need to show that
\begin{equation*}
	\level{\mathsf{I}}(M) \leq \level{\mathsf{I}}\Theta^1(M) + 1
\end{equation*}
and to do this, we need only look at the distinguished triangle
\begin{equation*}
	M \overset{\iota^0}{\longrightarrow} I^0 \longrightarrow \Theta^1(M) \longrightarrow
\end{equation*}
and recognize that $I^0$ is a bounded complex of injective modules with zero differential. Lemma~\ref{BasicLemma} then gives the necessary inequality.	
\end{proof}

\begin{lemma}\label{Splice}
Let $M \in \Db(R)$. For every integer $n \geq 1$ we have exact sequences
\begin{equation*}
	0 \longrightarrow \HH(\Omega^n(M)) \longrightarrow F_{n-1} \longrightarrow \cdots \longrightarrow F_0 \longrightarrow \HH(M) \longrightarrow 0
\end{equation*}
\begin{equation*}
	0 \longrightarrow \HH(M) \longrightarrow I^0 \longrightarrow I^1 \cdots \longrightarrow I^{n-1} \longrightarrow \HH(\Theta^n(M)) \longrightarrow 0
\end{equation*}
\end{lemma}
\begin{proof}
This follows from splicing together the short exact sequences on homology constructed earlier this section.
\end{proof}

\subsection{Homological and Gorenstein Homological Dimensions}
The development of projective, injective, and flat modules in commutative algebra is fundamental in the field, as is the construction of projective, injective, and flat resolutions of arbitrary complexes. For this paper, we need the following classical bounds on projective, injective, and flat dimensions of complexes in short exact sequences.

\begin{lemma}\label{DimBound}
Let $0 \longrightarrow L \longrightarrow M \longrightarrow N \longrightarrow 0$ be a short exact sequence of $R$-complexes. If any two of the three complexes has finite projective, injective, or flat dimension, then so does the third. In addition the following inequalities hold:
\begin{align*}
	\pd(L) &\leq \max(\pd(M), \pd(N) - 1) \\
	\fd(L) &\leq \max(\fd(M), \fd(N) - 1) \\
	\id(N) &\leq \max(\id(L)-1, \id(M))
\end{align*}
\end{lemma}
\begin{proof}
These are standard results in homological algebra - for reference see Chapter 8 in \cite{DCMCA}
\end{proof}

The development of Gorenstein projective, Gorenstein injective, and Gorenstein flat modules in commutative algebra is more recent. These modules are generalizations of classical projective, injective, and flat modules, and so occasionally have weaker homological properties. However, many properties of standard homological dimensions do pass over to Gorenstein homological dimensions. One example is a Gorenstein analogue of the previous lemma.

\begin{lemma}\label{GdimBound}
Let $0 \longrightarrow L \longrightarrow M \longrightarrow N \longrightarrow 0$ be a short exact sequence of $R$-complexes. If any two of the three complexes has finite Gorenstein projective, injective, or flat dimension, then so does the third. In addition, the following inequalities hold:
\begin{align*}
	\Gpd(L) &\leq \max(\Gpd(M), \Gpd(N) - 1) \\
	\Gfd(L) &\leq \max(\Gfd(M), \Gfd(N) - 1) \\
	\Gid(N) &\leq \max(\Gid(L)-1, \Gid(M))
\end{align*}
\end{lemma}
\begin{proof}
Propositions 9.1.16, 9.2.15, and 9.3.25 in \cite{DCMCA} provide the proofs. Note that 9.3.25 requires that $R$ be (right) Noetherian which is covered by our blanket assumption on $R$.
\end{proof}

\section{Level Inequalities}

We are almost ready to construct bounds for the homological levels and Gorenstein homological levels of any bounded $R$-complex. We need only the following standard homological constructions, and a result bounding the homological and Gorenstein homological dimensions of complexes.

\begin{note}
Let $M$ be a complex with zero differential. We set $M^\oplus = \bigoplus\limits_{i \in \mathbb{Z}} M_i$. Note that $\level{\Cl}(M) = \level{\Cl}(M^\oplus)$ for any collection $\Cl$ in $\sD(R)$, but that
\begin{equation*}
\Cld(M^\oplus) = \sup_{i \in \mathbb{Z}}\Cld(M_i)
\end{equation*}
where $\Cld$ denotes any homological or Gorenstein homological dimension. (Note that the above equality only holds for injective, Gorenstein injective, or Gorenstein flat dimensions if $R$ is Noetherian.)
\end{note}
\begin{note} \label{Gacc}
For any $R$-complex $M$ and for all $i \in \mathbb{Z}$, we have
\begin{gather} 
	\label{acc1}
	0 \longrightarrow \HH_i(M) \longrightarrow \C_i(M) \longrightarrow \B_{i-1}(M) \longrightarrow 0 \\
	\label{acc2}
	0 \longrightarrow \B_i(M) \longrightarrow \Z_i(M) \longrightarrow \HH_i(M) \longrightarrow 0 \\
	\label{acc3}
	0 \longrightarrow \B_i(M) \longrightarrow M_i \longrightarrow \C_i(M) \longrightarrow 0 \\
	\label{acc4}
	0 \longrightarrow \Z_i(M) \longrightarrow M_i \longrightarrow \B_{i-1}(M) \longrightarrow 0	
\end{gather}
\end{note}

We now prove generalizations of Lemma 3.6 and Proposition 3.7 in \cite{GPerf}, and extend them to all bounded complexes. It is important to note that the proofs given in \cite{GPerf} naturally translate to complexes with non-finitely generated homology and to all standard homological and Gorenstein homological dimensions with some adjustments.

\begin{lemma}
Let $M$ be an right-bounded $R$-complex (i.e., $M_i = 0$ for all $i \ll 0$). If $\HH_i(M)$ and $M_i$ have finite projective, flat, Gorenstein projective, or Gorenstein flat dimension for all $i \in \mathbb{Z}$, then $\B_i(M)$, $\Z_i(M)$, and $\C_i(M)$ all have correspondingly finite homological or Gorenstein homological dimension.
\end{lemma}
\begin{proof}
For the rest of this proof, let ``dimension'' refer to any of the four possible homological and Gorenstein homological dimensions. Let $u = \inf(M)$. We have $\HH_{u}(M) \cong \C_u(M)$, and so by \ref{acc3} and Lemmas~\ref{DimBound} and \ref{GdimBound}, we have that $\B_u(M)$ and $\Z_u(M)$ have finite dimension. By \ref{acc4}, we have that $\Z_{u+1}(M)$ has finite dimension, and thus we have that $\B_{u+1}(M)$ has finite dimension by \ref{acc2}. In addition, we have $\C_{u+1}(M)$ has finite dimension by \ref{acc3}. Continuing this process iteratively, we have that $\C_i(M)$, $\B_i(M)$, and $\Z_i(M)$ have finite dimension for all $i \in \mathbb{Z}$.
\end{proof}
\begin{lemma}
Let $M$ be a left-bounded $R$-complex (i.e., $M_i = 0$ for all $i \gg 0$). If $\HH_i(M)$ and $M_i$ have finite injective or Gorenstein injective dimension for all $i \in \mathbb{Z}$, then $\B_i(M)$, $\Z_i(M)$, and $\C_i(M)$ all have finite injective or Gorenstein injective dimension, respectively.
\end{lemma}
\begin{proof}
For the rest of this proof, let ``dimension'' refer to either injective or Gorenstein injective dimension. Let $s = \sup(M)$. We have $\HH_s(M) \cong \Z_s(M)$, $\B_s(M) = 0$, and $M_s \cong \C_s(M)$. By \ref{acc4} and Lemmas~\ref{DimBound} and \ref{GdimBound}, we have that $\B_{s-1}(M)$ has finite dimension. By Lemmas~\ref{DimBound} and \ref{GdimBound} applied to \ref{acc2} and \ref{acc3}, we have that $\C_{s-1}(M)$ and $\Z_{s-1}(M)$ have finite dimension. Continuing this process iteratively, we have that $\C_i(M)$, $\B_i(M)$ and $\Z_i(M)$ have finite dimension for all $i \in \mathbb{Z}$.
\end{proof} 

\begin{lemma}\label{FiniteDim}
Let $M \in \sD(R)$, and suppose that $\HH(M)^\oplus$ has finite homological or Gorenstein homological dimension. Then $M$ has correspondingly finite homological or Gorenstein homological dimension.
\end{lemma}
\begin{proof}
First consider the case when $\HH(M)^\oplus$ has finite projective, flat, Gorenstein projective, or Gorenstein flat dimension (referred to as dimension for the rest of the paragraph). Let $F$ be a semi-free resolution of $M$. By the previous lemma, we have that $\C_i(F)$ has finite dimension for all $i \in \mathbb{Z}$. Since the homology of $M$ vanishes for large indexes, we then have that $M$ has finite dimension.

Next consider the case where $M$ has finite injective or Gorenstein flat dimension (referred to as dimension for the rest of the paragraph). Let $I$ be a semi-injective resolution of $M$. By the previous lemma, we have that $\Z_i(F)$ has finite dimension for all $i \in \mathbb{Z}$. Since the homology of $M$ vanishes for small indexes, we then have that $M$ has finite dimension.
\end{proof}
We now arrive at the main results of the paper. We note that the following result providing bounds on projective level was established in a general triangulated category containing enough projectives by Christensen in \cite[Prop. 4.7]{IdealTriang}. However, the method of proof below is different and applies specifically to the derived category of a ring.  

\begin{theorem}\label{PInequality}
Let $M \in \Db(R)$, and let $\Proj$ be the collection of all projective $R$-modules. The following inequality holds:
\begin{equation*}
	 \level\Proj M \leq \pd_R\HH(M)^\oplus + 1
\end{equation*}
\end{theorem}

\begin{proof}
(This proof is essentially a reformulation of \cite[Proposition 5.5]{HomPerf}. In that proof, the authors assumed that the homology of $M$ be finitely generated, but the argument also works for complexes with non-finitely generated homologies.)

We may assume that $M$ is nonzero or else the statement is trivial. If the projective dimension of $\HH(M)^\oplus$ is infinite, there is nothing to prove, and so we assume that the projective dimension of $\HH(M)^\oplus$ is finite. We prove this theorem by induction on the projective dimension of $\HH(M)^\oplus$.

For the base case, we assume $\pd\HH(M)^\oplus = 0$. Thus, $\HH(M)^\oplus$ is a projective module. The canonical surjection of $R$-modules $\Z(M)^\oplus \to \HH(M)^\oplus$ then splits to form a map $\HH(M)^\oplus \to \Z(M)^\oplus$. This splitting map can be considered as a map of complexes $\sigma: \HH(M) \to \Z(M)$ since both $\HH(M)$ and $\Z(M)$ have zero differential. Composing $\sigma$ with the inclusion of complexes $\Z(M) \hookrightarrow M$ yields a quasi-isomorphism $\HH(M) \simeq M$. Thus, $\level{\Proj}(\HH(M)) = \level{\Proj}(M)$. $\HH(M)$ is a bounded complexes of projective modules with zero differential, and so $\level{\Proj}(\HH(M)) = 1$. Thus, $\level{\Proj}(M) = 1$ which establishes the base case.

For the inductive case, suppose that $\pd(\HH(M)^\oplus) = p$ for some positive integer $p$. Assume that for all complexes $L \in \Db(R)$ such that $\pd(\HH(L)^\oplus) < p$, we have that $\level{\Proj}(L) \leq \pd(\HH(L)^\oplus) + 1$. Consider the complex $S = \Omega_R^{1}(M)$. By Lemma~\ref{Splice}, we have that $\pd(\HH(S)^\oplus) = p-1$, or else the projective dimension of $\HH(M)^\oplus$ would be strictly less than $p$. By the inductive hypothesis, we have that $\level{\Proj}(S) \leq \pd(\HH(S)^\oplus) + 1 = p$. By Lemma~\ref{SyzygyInequality}, we have that
\begin{align*}
	\level{\Proj}(M) 	&\leq \level{\Proj}(S) + 1 \\
						&\leq p + 1 \\
						&\leq \pd(\HH(M)^\oplus) + 1	
\end{align*}
which completes the inductive case.
\end{proof}

We can form a injective version of the above proposition using a similar proof. 

\begin{theorem}\label{IInequality}
Let $M \in \Db(R)$, and let $\Inj$ be the collection of all injective $R$-modules. The following inequality holds:	
\begin{equation*}
	\level\Inj M \leq \id\HH(M)^\oplus + 1
\end{equation*}
\end{theorem}
\begin{proof}
We may assume that $M$ is nonzero or else the statement is trivial. If the injective dimension of $\HH(M)^\oplus$ is infinite, there is nothing to prove, and so we assume that the injective dimension of $\HH(M)^\oplus$ is finite. We prove this theorem by induction on the injective dimension of $\HH(M)^\oplus$.

For the base case, we assume $\id\HH(M)^\oplus = 0$. Thus, $\HH(M)^\oplus$ is an injective module. The canonical injection of $R$-modules $\HH(M)^\oplus \hookrightarrow \C(M)^\oplus$ then splits to form a map $\C(M)^\oplus \to \HH(M)^\oplus$. This splitting map can be considered as a map of complexes $\sigma: \C(M) \to \HH(M)$ since both $\HH(M)$ and $\C(M)$ have zero differential. Composing $\sigma$ with the morphism $M \twoheadrightarrow \C(M)$ yields a quasi-isomorphism $M \to \HH(M)$. Thus, $\level{\Inj}(\HH(M)) = \level{\Inj}(M)$. $\HH(M)$ is a bounded complexes of injective modules with zero differential, and so $\level{\Inj}(\HH(M)) = 1$. Thus, $\level{\Inj}(M) = 1$ which establishes the base case.

For the inductive case, suppose that $\id(\HH(M)^\oplus) = i$ for some positive integer $i$. Assume that for all complexes $L \in \Db(R)$ such that $\id(\HH(L)^\oplus) < i$, we have that $\level{\Inj}(L) \leq \id(\HH(L)^\oplus) + 1$. Consider the complex $T = \Theta^{1}(M)$. By Lemma~\ref{Splice}, we have that $\id(\HH(T)^\oplus) = i-1$, or else the injective dimension of $\HH(M)^\oplus$ would be strictly less than $i$. By the inductive hypothesis, we have that $\level{\Inj}(T) \leq \id(\HH(T)^\oplus) + 1 = i$. By Lemma~\ref{SyzygyInequality}, we have that
\begin{align*}
\level{\Inj}(M) 	&\leq \level{\Inj}(T) + 1 \\
					&\leq i + 1 \\
					&\leq \id(\HH(M)^\oplus) + 1	
\end{align*}
which completes the inductive case.
\end{proof}

In order to obtain the bounds in the previous two results, the existence of a splitting map involving $\HH(M)^\oplus$ was crucial. Such a splitting map does not generally occur when dealing with flat modules, or with Gorenstein projective, injective, or flat modules. Thus, we can only prove a slightly weaker bound in these cases. Because the proofs with respect to flat dimension, Gorenstein projective dimension, and Gorenstein flat dimension are so similar, we combine their proofs in a single theorem.

\begin{theorem} \label{CInequality}
Let $M \in \Db(R)$, and let $\Cl$ be the class of flat, Gorenstein projective, or Gorenstein flat modules. Then the following inequality holds:
\begin{equation*}
	\level{\Cl}(M) \leq \max\{2, \Cld(\HH(M)^\oplus) + 1\}
\end{equation*}
where $\Cld$ represents flat dimension, Gorenstein projective dimension, or Gorenstein flat dimension, respectively.
\end{theorem}
\begin{proof}
We may assume that $M$ is nonzero or else the statement is trivial. If the $\Cl$-dimension of $\HH(M)^\oplus$ is infinite, there is nothing to prove, and so we assume that the $\Cl$-dimension of $\HH(M)^\oplus$ is finite. Since the $\Cl$-dimension of $\HH(M)^\oplus$ is finite, the $\Cl$-dimension of $M$ is finite by Lemma~\ref{FiniteDim}. Thus, $M$ is isomorphic in $\sD(R)$ to a bounded complex of $\Cl$-modules $C$ with $C_i = 0$ for $i < \inf(M)$ or $i > \Cld(M)$. For the rest of this proof, we replace $M$ with $C$.

We now split the proof into two cases: one where the $\Cl$-dimension of $\HH(M)^\oplus$ is zero, and one where the $\Cl$-dimension of $\HH(M)^\oplus$ is positive.

For the first case, we assume $\Cld(\HH(M)^\oplus) = 0$. Thus, each non-zero homology module of $M$ is a $\Cl$-module. Let $u = \inf(M)$. Since $M_{u-1} = 0$, we have that $\B_{u-1}(M) = 0$, and so $\HH_u(M) \cong \C_u(M)$. Consider the resulting short exact sequence obtained from \ref{acc3}, replacing $\C_u(M)$ with $\HH_u(M)$.
\begin{equation*}
	0 \longrightarrow \B_u(M) \longrightarrow M_u \longrightarrow \HH_u(M) \longrightarrow 0
\end{equation*}
Since both $M_u$ and $\HH_u(M)$ are $\Cl$-modules, we have that $\B_u(M)$ is a $\Cl$-module by Lemmas~\ref{DimBound} or \ref{GdimBound}. (Keep in mind that the $\C$-dimension of a module must be non-negative.) Now by using \ref{acc2}, we have the short exact sequence
\begin{equation*}
	0 \longrightarrow \B_u(M) \longrightarrow \Z_u(M) \longrightarrow \HH_u(M) \longrightarrow 0
\end{equation*}
which forces $\Z_u(M)$ to be a $\Cl$-module, again by Lemmas~\ref{DimBound} or \ref{GdimBound}. Thus, we have that $\Z_u(M)$ and $\B_u(M)$ are $\Cl$-modules.

Now consider the short exact sequence obtained from \ref{acc4}
\begin{equation*}
	0 \longrightarrow \Z_{u+1}(M) \longrightarrow M_{u+1} \longrightarrow \B_u(M) \longrightarrow 0
\end{equation*}
Since $\B_u(M)$ has been shown to be a $\Cl$-module and $M_{u+1}$ is a $\Cl$-module by assumption, another application of Lemmas~\ref{DimBound} or \ref{GdimBound} yields that $\Z_{u+1}(M)$ is a $\Cl$-module. This implies that $\B_{u+1}(M)$ is a $\Cl$-module by \ref{acc2} and Lemmas~\ref{DimBound} or \ref{GdimBound}. Thus, $\B_{u+1}(M)$ and $\Z_{u+1}(M)$ are $\Cl$-modules. Repeating this procedure for all integers greater than $u$ yields that $\B(M)^\oplus$ and $\Z(M)^\oplus$ are $\Cl$-modules. Since $\B(M)$ and $\Z(M)$ are complexes with zero differential, we then have that $\level{\Cl}(\B(M)) = \level{\Cl}(\Z(M)) = 1$. We have the distinguished triangle in $\sD(R)$.
\begin{equation*}
	\Z(M) \longrightarrow M \longrightarrow \Sigma\B(M) \longrightarrow
\end{equation*}
which implies that $\level{\Cl}(M) \leq 2$ by Lemma~\ref{BasicLemma}. This proves the first case.
		
For the second case, suppose that $\Cld_R(\HH(M)^\oplus) = n$ for some positive integer $n$. Consider the complex $S = \Omega_R^{n-1}(M)$.  By the construction of the Adams resolution and Lemmas~\ref{DimBound} or \ref{GdimBound}, we have that $S$ has finite $\Cl$-dimension, and so can be replaced by a bounded complex of $\Cl$-modules in a similar manner as in the beginning of the proof. By Lemma~\ref{Splice}, we have that $\Cld_R(\HH(S)^\oplus) = 1$, or else the $\Cl$-dimension of $\HH(M)^\oplus$ would be strictly less than $n$. Let $u = \inf(S)$. Since $S_{u-1} = 0$, we have that $\B_{u-1}(S) = 0$, and so $\HH_u(S) \cong \C_u(S)$. Consider the resulting short exact sequence obtained from \ref{acc3}, replacing $\C_u(S)$ with $\HH_u(S)$.
\begin{equation*}
	0 \longrightarrow \B_u(S) \longrightarrow S_u \longrightarrow \HH_u(S) \longrightarrow 0
\end{equation*}
Since $\Cld(S_u) = 0$ and $\Cld(\HH_u(S))$ is zero or one, we have that $\B_u(S)$ is a $\Cl$-module by Lemmas~\ref{DimBound} or \ref{GdimBound}. Since $S_{u-1} = 0$, we have $S_u \cong Z_u$ and so $\Z_u(S)$ is a $\Cl$-module. Thus, we have that $\Z_u(S)$ and $\B_u(S)$ are $\Cl$-modules.

Now consider the short exact sequence obtained from \ref{acc4}
\begin{equation*}
	0 \longrightarrow \Z_{u+1}(S) \longrightarrow S_{u+1} \longrightarrow \B_u(S) \longrightarrow 0
\end{equation*}
Since $\B_u(S)$ has been shown to be a $\Cl$-module and $S_{u+1}$ is a $\Cl$-module by assumption, another application of Lemmas~\ref{DimBound} or \ref{GdimBound} yields that $\Z_{u+1}(S)$ is a $\Cl$-module. This implies that $\B_{u+1}(S)$ is a $\Cl$-module by \ref{acc2} and Lemmas~\ref{DimBound} or \ref{GdimBound}. Thus, $\B_{u+1}(S)$ and $\Z_{u+1}(S)$ are $\Cl$-modules. Repeating this procedure for all integers greater than $u$ yields that $\B(S)^\oplus$ and $\Z(S)^\oplus$ are $\Cl$-modules. Since $\B(S)$ and $\Z(S)$ are complexes with zero differential, we then have that $\level{\Cl}(\B(S)) = \level{\Cl}(\Z(S)) = 1$. We have the distinguished triangle in $\sD(R)$.
\begin{equation*}
	\Z(S) \longrightarrow S \longrightarrow \Sigma\B(S) \longrightarrow
\end{equation*}
which implies that $\level{\Cl}(S) \leq 2$ by Lemma~\ref{BasicLemma}. We then have
\begin{align*}
	\level{\Cl}(M) 	&\leq \level{\Cl}(S) + (n-1) \\
					&\leq 2 + (n-1) \\
					&\leq \Cld(\HH(M)^\oplus) + 1		
\end{align*}
where the first inequality holds by Lemma~\ref{SyzygyInequality}. This completes the second case.
\end{proof}

Finally, we bound the level of a bounded complex with respect to the class of Gorenstein injective modules.

\begin{theorem} \label{GIInequality}
Let $M \in \Db(R)$, and let $\GI$ be the collection of Gorenstein injective modules. Then the following inequality holds:
\begin{equation*}
	\level{\GI}(M) \leq \max\{2, \Gid(\HH(M)^\oplus) + 1\}
\end{equation*}
\end{theorem}
\begin{proof}
We may assume that $M$ is nonzero or else the statement is trivial. If the Gorenstein injective dimension of $\HH(M)^\oplus$ is infinite, there is nothing to prove, and so we assume that the Gorenstein injective dimension of $\HH(M)^\oplus$ is finite. Since the Gorenstein injective dimension of $\HH(M)^\oplus$ is finite, the Gorenstein injective dimension of $M$ is finite by Lemma~\ref{FiniteDim}. Thus, $M$ is isomorphic in $\sD(R)$ to a bounded complex of Gorenstein injective modules $I$ with $I_i = 0$ for $i < -\Gid(M)$ or $i > \sup(M)$. For the rest of this proof, we replace $M$ with $I$.

We now split the proof into two cases: one where the Gorenstein injective dimension of $\HH(M)^\oplus$ is zero, and one where the Gorenstein injective dimension of $\HH(M)^\oplus$ is positive.

For the first case, we assume $\Gid(\HH(M)^\oplus) = 0$. Thus, each non-zero homology module of $M$ is a Gorenstein injective module. Let $s = \sup(M)$. Since $M_{s+1} = 0$, we have that $\B_{s}(M) = 0$ and $M_s \cong \C_s(M)$. This forces both $\B_s(M)$ and $\C_s(M)$ to be Gorenstein injective. Consider the resulting short exact sequence obtained from \ref{acc1}
\begin{equation*}
	0 \longrightarrow \HH_s(M) \longrightarrow \C_s(M) \longrightarrow \B_{s-1}(M) \longrightarrow 0
\end{equation*}
Since both $\C_s(M)$ and $\HH_s(M)$ are Gorenstein injective, we have that $\B_{s-1}(M)$ is Gorenstein injective by Lemma~\ref{GdimBound}. (Keep in mind that the Gorenstein injective dimension of a module must be nonnegative.) Now by using \ref{acc3}, we have the short exact sequence
\begin{equation*}
	0 \longrightarrow \B_{s-1}(M) \longrightarrow M_{s-1} \longrightarrow \C_{s-1}(M) \longrightarrow 0
\end{equation*}
which forces $\C_{s-1}(M)$ to be a Gorenstein injective module, again by Lemma~\ref{GdimBound}. Thus, we have that $\C_{s-1}(M)$ and $\B_{s-1}(M)$ are Gorenstein injective modules.

Through repeated applications of Lemma~\ref{GdimBound} with \ref{acc1} and \ref{acc3}, we have that $\C_i(M)$ and $\B_i(M)$ are Gorenstein injective modules for all $i \in \mathbb{Z}$. Since $\B(M)$ and $\C(M)$ are complexes with zero differential, we then have that $\level{\GI}(\B(M)) = \level{\GI}(\C(M)) = 1$. We have the distinguished triangle in $\sD(R)$.
\begin{equation*}
	\B(M) \longrightarrow M \longrightarrow \C(M) \longrightarrow
\end{equation*}
which implies that $\level{\GI}(M) \leq 2$ by Lemma~\ref{BasicLemma}. This proves the first case.
		
For the second case, suppose that $\Gid_R(\HH(M)^\oplus) = n$ for some positive integer $n$. Consider the complex $T = \Theta_R^{n-1}(M)$.  By the construction of the injective Adams resolution and Lemma~\ref{GdimBound}, we have that $T$ has finite Gorenstein injective dimension, and so can be replaced by a bounded complex of Gorenstein injective modules in a similar manner as in the beginning of the proof. By Lemma~\ref{Splice}, we have that $\Gid_R(\HH(T)^\oplus) = 1$, or else the Gorenstein injective dimension of $\HH(M)^\oplus$ would be strictly less than $n$. Let $s = \sup(T)$. Since $T_{s+1} = 0$, we have that $\B_{s}(T) = 0$, and $T_s \cong \C_s(T)$. This implies that $\B_{s}(T)$ and $\C_{s}(T)$ are Gorenstein injective. Consider the resulting short exact sequence obtained from \ref{acc1}
\begin{equation*}
	0 \longrightarrow \HH_s(T) \longrightarrow \C_s(T) \longrightarrow \B_{s-1}(T) \longrightarrow 0
\end{equation*}
Since $\Gid(\C_s(T)) = 0$ and $\Gid(\HH_s(T))$ is zero or one, we have that $\B_{s-1}(T)$ is a Gorenstein injective module by Lemma~\ref{GdimBound}. Now by using \ref{acc3}, we have the short exact sequence
\begin{equation*}
	0 \longrightarrow \B_{s-1}(T) \longrightarrow T_{s-1} \longrightarrow \C_{s-1}(T) \longrightarrow 0
\end{equation*}
which forces $\B_{s-1}(T)$ to be a Gorenstein injective module, again by Lemma~\ref{GdimBound}. Thus, we have that $\C_{s-1}(T)$, and $\B_{s-1}(T)$ are Gorenstein injective modules.

Through repeated applications of Lemma~\ref{GdimBound} with \ref{acc1} and \ref{acc3}, we have that $\C_i(T)$ and $\B_i(T)$ are Gorenstein injective modules for all $i \in \mathbb{Z}$. Since $\B(T)$ and $\C(T)$ are complexes with zero differential, we then have that $\level{\GI}(\B(T)) = \level{\GI}(\C(T)) = 1$. We have the distinguished triangle in $\sD(R)$.
Since $\B(T)$ and $\Z(T)$ are complexes with zero differential, we then have that $\level{\GI}(\B(T)) = \level{\GI}(\Z(T)) = 1$. We have the distinguished triangle in $\sD(R)$.
\begin{equation*}
	\Z(T) \longrightarrow T \longrightarrow \C(T) \longrightarrow
\end{equation*}
which implies that $\level{\GI}(T) \leq 2$ by Lemma~\ref{BasicLemma}. We then have
\begin{align*}
	\level{\GI}(M) 	&\leq \level{\GI}(T) + (n-1) \\
					&\leq 2 + (n-1) \\
					&\leq \Gid(\HH(M)^\oplus) + 1		
\end{align*}
where the first inequality holds by Lemma~\ref{SyzygyInequality}. This completes the second case.
\end{proof}

We now have the following level inequalities for all bounded complexes over regular and Gorenstein local rings.

\begin{corollary}\label{Regular}
Let $(R, \mx, k)$ be a noetherian regular local ring of dimension $d$. The following inequalities hold.
\begin{enumerate}
\item $\sup\{\level{\Proj}(M)\, | \, M \in \Db(R)\} \leq d+1$
\item $\sup\{\level{\Inj}(M)\, | \, M \in \Db(R)\} \leq d+1$
\item $\sup\{\level{\Flat}(M)\, | \, M \in \Db(R)\} \leq d+1$ 
\end{enumerate}
\end{corollary}
\begin{proof}
The first two inequalities immediately follow from Theorem~\ref{PInequality} and Theorem~\ref{IInequality}. If the dimension of $R$ is positive, then the third inequality immediately follows from Theorem~\ref{CInequality}. If $d = 0$, then $R$ is a field, and any bounded complex over a field has a $\Flat$-level of at most 1 since every complex over a field is quasi-isomorphic to its homology \cite[Cor. 4.2.18]{DCMCA}.
\end{proof}

\begin{corollary}\label{Gorenstein}
Let $(R, \mx, k)$ be a noetherian Gorenstein local ring of dimension $d$. The following inequalities hold.
\begin{enumerate}
\item $\sup\{\level{\GP}(M)\, | \, M \in \Db(R)\} \leq \max\{2, d+1\}$
\item $\sup\{\level{\GI}(M)\, | \, M \in \Db(R)\} \leq \max\{2, d+1\}$
\item $\sup\{\level{\GF}(M)\, | \, M \in \Db(R)\} \leq \max\{2, d+1\}$
\end{enumerate}
\end{corollary}
\begin{proof}
Immediate from Theorems~\ref{CInequality} and \ref{GIInequality}.
\end{proof}

\section{The Bass Formula for Injective and Gorenstein Injective Levels}

As an application of the previous theorems, we can prove the Bass Formula for injective level, and a version of the Bass Formula for Gorenstein injective level. To prove these results, we also need injective and Gorenstein injective versions of \cite[Thm. 2.1]{LowerBound}, which have recently been established in \cite{LB2}.

\begin{theorem}\label{Bass}
Let $R$ be a noetherian local ring, and let $M \in \Dbf(R)$ such that $\id(\HH(M)^\oplus)) < \infty$. We have that
\begin{equation*}
\level{\Inj}(M) = \depth(R)+1
\end{equation*}
\end{theorem}
\begin{proof}
By Theorem~\ref{IInequality}, we have that $\level{\Inj}(M) \leq \id(\HH(M)^\oplus) + 1$. Since $\HH(M)^\oplus$ is a finitely generated module of finite injective dimension, the Bass Formula for modules implies that $\id(\HH(M)^\oplus) = \depth(R)$. Thus,
\begin{equation*}
\level{\Inj}(M) \leq \depth(R) + 1
\end{equation*}

For the opposite inequality, first note that by Lemma~\ref{FiniteDim}, we have that $\id(M) < \infty$. Suppose that $\id(M) = -\inf(M)$. Then by \cite[Cor. 16.4.11]{DCMCA}, we have that $\depth(R) = 0$. We always have that $\level{\Inj}(M) \geq 1$ (since the zero complex has infinite injective dimension), and so $\level{\Inj}(M) = \depth(R) + 1$.

Otherwise, we have that for all $-\id(M) < i < \inf(M)$, $\HH_i(M) = 0$, and that $\HH_0(\Upsilon^{-\id(M)+1}(M))$ cannot be injective or else the injective dimension of $M$ would be strictly smaller. Thus, we have
\begin{equation*}
\level{\Inj}(M) \geq \inf(M) + \id(M) + 1
\end{equation*}
by \cite[Cor. 4.3]{LB2}. But by \cite[Cor. 16.4.11]{DCMCA}, we have that $\inf(M) + \id(M) = \depth(R)$, and so
\begin{equation*}
\level{\Inj}(M) \geq \depth(R) + 1
\end{equation*}
Thus, we have
\begin{equation*}
\level{\Inj}(M) = \depth(R) + 1
\end{equation*}
\end{proof}

It is natural to wonder if the assumption $\id(\HH(M)^\oplus) < \infty$ in the above theorem can be weakened to be that $\level{\Inj}(M) < \infty$. This cannot be done owing to a counterexample very similar to Example 3.10 in \cite{LevelGor}. Consider the ring $R = k[x]/(x^2)$ (where $k$ is any field), and let $K$ be the Koszul complex on the element $x$. Since $R$ is injective, we have that $K$ is a bounded semi-injective complex, and thus has finite injective level. If the injective level of $K$ is 1, then there would exist a quasi-isomorphism from $K$ to $\HH(K)$. But since $K_0 = R$ is cyclic, any module homomorphism $K_0 \to \HH_0(K)$ is multiplication by some element $t \in \HH_0(K)$, which induces the zero map on homology. Thus, the injective level of $K$ is (at least) two, which is larger than $\depth(R) + 1 = 1$. It is important to note in this example that $\id(\HH_0(K)) = \infty$ since $R$ is not regular.

\begin{theorem}\label{GIBass}
Let $R$ be a noetherian local ring of positive depth, and let $M \in \Dbf(R)$ such that $\Gid(\HH(M)^\oplus) < \infty$. We have that
\begin{equation*}
\level{\GI}(M) = \depth(R) + 1
\end{equation*}
\end{theorem}

\begin{proof}
By Theorem~\ref{GIInequality}, we have that $\level{\GI}(M) \leq \max\{2, \Gid(\HH(M)^\oplus) + 1\}$. Since $\HH(M)^\oplus$ is a finitely generated module of finite Gorenstein injective dimension, the Bass Formula for Gorenstein injective dimension (see \cite[Cor. 19.2.14]{DCMCA}) implies that $\Gid(\HH(M)^\oplus) = \depth(R)$. Since the depth of $R$ is positive, we have that
\begin{equation*}
\level{\GI}(M) \leq \depth(R) + 1
\end{equation*}

For the opposite inequality, we first note that by Lemma~\ref{FiniteDim}, we have that $\Gid(M) < \infty$. By \cite[Cor. 4.3]{LB2}, we have that
\begin{equation*}
\level{\GI} M \geq \Gid(M) + \inf(M) + 1
\end{equation*}
But by \cite[Theorem 19.2.40]{DCMCA}, we have that $\Gid(M) + \inf(M) = \depth(R)$, and so
\begin{equation*}
\level{\GI}(M) \geq \depth(R) + 1
\end{equation*}
Thus, we have
\begin{equation*}
\level{\GI} = \depth(R) + 1
\end{equation*}
\end{proof}
For an artinian ring, no such Bass Formula for Gorenstein injective dimension holds. As an example of this failure, consider the ring $R = k[x]/(x^2)$ (where $k$ is any field). Over $R$, the residue field $k$ is an Gorenstein injective module, as $R$ is artinian Gorenstein. Thus, $\level{\GI}(k) = 1$. Let $K$ be the Koszul complex over $x$. The homology modules of $K$ are Gorenstein injective since $R$ is an artinian Gorenstein ring. Thus, $\Gid(\HH(K)^\oplus) \leq 2$. If the Gorenstein injective level of $K$ is 1, then $K$ would be isomorphic to its homology in the derived category. But since $K$ is a bounded semi-injective complex, this would imply that $K$ is quasi-isomorphic to its homology. Since $K_0 = R$ is cyclic, any module homomorphism $K_0 \to \HH_0(K)$ is multiplication by some element $t \in \HH_0(K)$, which induces the zero map on homology. Thus, the Gorenstein injective level of $K$ is 2.

\section{Optimality of Upper Bounds for Levels}
The upper bounds obtained on all homological levels and Gorenstein homological levels considered in this paper are optimal. Using previously established results, it is fairly straightforward to see that the upper bounds obtained with respect to projective, injective, Gorenstein projective, and Gorenstein injective level are achieved.
\begin{itemize}
\item (Projective and Injective Level) By \cite[Cor. 2.2]{LowerBound}, we have that for a module $M$, $\level{\Proj}(M) = \pd(M) + 1$, and by a dual argument to \cite[Cor. 2.2]{LowerBound}, we have that for a module $M$, $\level{\Inj}(M) = \id(M) + 1$. Thus, by choosing a module of finite projective or injective dimension (for instance, the residue field of a regular local ring), the upper bound established is obtained.
\item (Gorenstein Projective Level) By \cite[Cor. 3.5]{LevelGor}, we have that for a module $\level{\GP}(M) = \Gpd(M)+1$. Thus, by choosing a module of finite nonzero Gorenstein projective dimension (for instance, the residue field of Gorenstein local ring of positive dimension), the upper bound established is obtained when considering complexes whose homologies have positive Gorenstein projective dimension. Example 3.10 in \cite{LevelGor} shows the necessity of allowing for a complex with Gorenstein projective homologies but with $\GP$-level two.
\item (Gorenstein Injective Level) By \cite[Cor. 4.3]{LB2}, we have that for a module $M$ $\level{\GI}(M) \geq \Gid(M) + 1$, and thus by an argument dual to \cite[Cor. 3.5]{LevelGor}, we have that $\level{\GI}(M) = \Gid(M) + 1$. Thus, by choosing a module of finite nonzero Gorenstein injective dimension (for instance, the residue field of a Gorenstein local ring of positive dimension), the upper bound established is achieved when considering complexes whose homologies have positive Gorenstein injective dimension. The example after Theorem~\ref{GIBass} shows the necessity of allowing for a complex with Gorenstein injective homologies with $\GI$-level two.
\end{itemize}

Establishing the optimality of the upper bounds for flat and Gorenstein flat dimension is more difficult but still possible. We first need some lemmas that provide lower bounds for the flat and Gorenstein flat levels of modules.

\begin{lemma}
Let $M$ be an $R$-module. Then $\level{\Flat}(M) \geq \fd(M)+1$.
\end{lemma}
\begin{proof}
By applying \cite[Lemma 2.4.6]{HomPerf} with respect to the functor $\Hom_R(-, \mathbb{E})$ we have the following inequality:
\begin{equation*}
\level{\Hom_R(\Flat, \mathbb{E})}(\Hom_R(M, \mathbb{E})) \leq \level{\Flat}(M)
\end{equation*}
However, for any flat module $F$ over $R$, we have that $\Hom_R(F, \mathbb{E})$ is injective by flat-injective duality. Thus, the set $\Hom_R(\Flat, \mathbb{E})$ is contained in the set $\Inj$. Thus, we have that
\begin{equation*}
\level{\Inj}(\Hom_R(M, \mathbb{E})) \leq \level{\Hom_R(F, \mathbb{E})}(\Hom_R(M, \mathbb{E}))
\end{equation*}
We have that $\level{\Inj}(\Hom_R(M, \mathbb{E})) = \id(\Hom_R(M, \mathbb{E}) + 1$ since $\Hom_R(M, \mathbb{E})$ is a module, and by flat-injective duality we have that $\fd(M) = \id(\Hom_R(M, \mathbb{E})$. Thus, we have the inequality
\begin{equation*}
\level{\Flat}(M) \geq \fd(M)+1
\end{equation*}
\end{proof}

\begin{lemma}
Let $M$ be an $R$-module. Then $\level{\GF}(M) \geq \Gfd(M)+1$.
\end{lemma}
\begin{proof}
By applying \cite[Lemma 2.4.6]{HomPerf} with respect to the functor $\Hom_R(-, \mathbb{E})$ we have the following inequality:
\begin{equation*}
\level{\Hom_R(\GF, \mathbb{E})}(\Hom_R(M, \mathbb{E})) \leq \level{\GF}(M)
\end{equation*}
However, for any Gorenstein flat module $F$ over $R$, we have that $\Hom_R(F, \mathbb{E})$ is Gorenstein injective by Gorenstein flat-Gorenstein injective duality \cite[Thm. 9.3.12]{DCMCA}. (Note that we assume that $R$ is noetherian.) Thus, the set $\Hom_R(\GF, \mathbb{E})$ is contained in the set $\GI$. Thus, we have that
\begin{equation*}
\level{\GI}(\Hom_R(M, \mathbb{E})) \leq \level{\Hom_R(F, \mathbb{E})}(\Hom_R(M, \mathbb{E}))
\end{equation*}
We have that $\level{\GI}(\Hom_R(M, \mathbb{E})) = \Gid(\Hom_R(M, \mathbb{E}) + 1$ since $\Hom_R(M, \mathbb{E})$ is a module, and by Gorenstein flat-Gorenstein injective duality we have that $\Gfd(M) = \Gid(\Hom_R(M, \mathbb{E})$ \cite[Prop. 9.3.20]{DCMCA}. Thus, we have the inequality
\begin{equation*}
\level{\GF}(M) \geq \Gfd(M)+1
\end{equation*}
\end{proof}

These two lemmas show that the upper bounds established in Theorem~\ref{CInequality} are optimal when considering complexes whose homologies have positive flat or Gorenstein flat dimension. In particular, the residue field of a regular local ring $R$ of positive dimension has flat and Gorenstein flat level of $\dim(R) + 1$.

It is possible, however, to construct a complex over a noetherian ring whose homologies are flat (and Gorenstein flat) but has flat level two (and Gorenstein flat level two), as the following example will show. This example was shared with us by L. Christensen, and uses Osofsky's construction of a flat module over a regular ring with projective dimension at least two.

\begin{example}
Let $R = \mathbb{R}[x, y, z]$, and let $Q$ be the field of fractions of $R$. By a result of Osofsky \cite{HomDim}, the projective dimension of $Q$ is at least two. Since $Q$ is the field of fractions of $R$, it is flat. Let $P$ be a free resolution of $Q$, and consider the hard truncation of $P$ at homological degree 1, denoted $P_{\leq 1}$. We have that $\HH_0(P_{\leq 1}) \cong Q$ and that $\HH_1(P_{\leq 1}) \cong \Omega_R^2(Q)$, and so the homology modules of $P_{\leq 1}$ are all flat (as all syzygies of flat modules are themselves flat). By Theorem~\ref{CInequality}, we have that $\level{\Flat}(P_{\leq 1}) \leq 2$. We claim that the flat level of $P_{\leq 1}$ is 2. Suppose for the sake of contradiction that $\level{\Flat}(P_{\leq 1}) = 1$. Then $P_{\leq 1}$ is isomorphic in the derived category to its homology. However, since $P_{\leq 1}$ is a bounded complex of projectives, we have that there exists a quasi-isomorphism between $P_{\leq 1}$ and its homology \cite[Prop. 6.4.20]{DCMCA}. This implies that there exists a chain map $\pi: P_{\leq 1} \to \HH(P_{\leq 1})$ inducing an isomorphism on homology:
\begin{center}
\begin{tikzcd}
0 \arrow[r] & P_1 \arrow[r]\arrow[d, "\pi_1"] & P_0 \arrow[r]\arrow[d, "\pi_0"] & 0 \\
0 \arrow[r] & \Omega^2_R(Q) \arrow[r, "0"] & Q \arrow[r] & 0
\end{tikzcd}
\end{center}
It is important to note that since the induced map $\Z(\pi)$ is a surjection, the map $\pi$ is itself a surjection (see for instance \cite[Lemma 4.2.7]{DCMCA}). Since $\pi$ is a quasi-isomorphism, we have that the mapping cone of $\pi$ is acyclic. Thus, there is an injection
\begin{equation*}
P_1 \hookrightarrow P_0 \oplus \Omega^2_R(Q)
\end{equation*}
where $p_1 \in P_1$ is sent to $(\partial(p_1), \pi_1(p_1)) \in P_0 \oplus \Omega_R^2(Q)$. Since $\pi_1$ is surjective, we have that $P_1 \cong \B_0(P_{\leq 1}) \oplus \Omega^2_R(Q)$. However, this makes $\Omega^2_R(Q)$ a direct summand of a projective module, and so $\Omega^2_R(Q)$ is projective. This forces the projective dimension of $Q$ to be at most one (see for instance \cite[Thm. 8.1.8]{DCMCA}). This is a contradiction, and so $\level{\Flat}(Q) = 2$. This example also shows the optimality of the upper bound for Gorenstein flat level, as over the regular ring  of finite Krull dimension $\mathbb{R}[x, y, z]$, every Gorenstein flat module is flat \cite[Cor. 9.3.22]{DCMCA}. This implies that the Gorenstein flat level is the same as the flat level of any complex. 
\end{example}

\subsection*{Acknowledgments}
The author would like to thank his thesis advisor, Tom Marley, for both suggesting the topic of this paper, and for the many insights and discussions that shaped this paper. The author would like to thank Lars Christensen and Andrew Soto Levins for inspiring the author to write the last two sections of the paper, and would like to thank Janina Letz for general comments on the paper and for pointing out that Theorem~\ref{PInequality} had establish in \cite{IdealTriang}. In addition, the author would like to thank Lars Christensen for suggesting Example 5.1, and Andrew Soto Levins for pointing out the existence of \cite{LB2} and for pointing out the Bass Formula for Gorenstein injective level.

\bibliographystyle{amsplain}
\bibliography{LevelInequalities.bib}

\providecommand{\bysame}{\leavevmode\hbox to3em{\hrulefill}\thinspace}
\providecommand{\MR}{\relax\ifhmode\unskip\space\fi MR }
\providecommand{\MRhref}[2]{%
  \href{http://www.ams.org/mathscinet-getitem?mr=#1}{#2}
}
\providecommand{\href}[2]{#2}
\begin{thebibliography}{1}

\bibitem{LowerBound}
Hannah Altmann, Elo{\'i}sa Grifo, Jonathan Monta{\~n}o, William Sanders, and
  Thanh Vu, \emph{Lower bounds of projective levels of complexes}, Journal of
  Algebra \textbf{491} (2017), 343--356.

\bibitem{HomPerf}
Luchezar Avramov, Ragnar-Olaf Buchweitz, Srikanth Iyengar, and Claudia Miller,
  \emph{Homology of perfect complexes}, Advances in Mathematics \textbf{223}
  (2010), 1731--1781.

\bibitem{LevelGor}
Laila Awadalla and Thomas Marley, \emph{Level and {Gorenstein} projective
  dimension}, Journal of Algebra \textbf{609} (2022), 606--618.

\bibitem{IdealTriang}
J.~Daniel Christensen, \emph{Ideals in triangulated categories: phantoms,
  ghosts and skeleta}, Advances in Mathematics \textbf{136} (1998), 284--339.

\bibitem{GPerf}
Lars Christensen, Antonia Kekkou, Justin Lyle, and Andrew~Soto Levins,
  \emph{G-levels of perfect complexes}, Available on arXiv at
  https://arxiv.org/abs/2507.10486v1, 2025.

\bibitem{DCMCA}
Lars Christensen, Hans-Bj\o rn~Foxby, and Hendrik Holm, \emph{Derived category
  methods in commutative algebra}, 2024.

\bibitem{LB2}
Yuki Mifune, \emph{Lower bounds for levels of complexes by resolution
  dimension}, Available on arXiv at https://arxiv.org/abs/2501.12109v1, 2025.

\bibitem{HomDim}
B.~L. Osofsky, \emph{Homological dimension and the continuum hypothesis},
  Transactions of the American Mathematical Society \textbf{132} (1968), no.~1,
  217--230.

\end{thebibliography}

\end{document}